\documentclass[11pt]{article}
\usepackage{amscd}
\usepackage{amsfonts}
\usepackage{amsmath}
\usepackage{amssymb}
\usepackage{amsthm}
\usepackage{bbm}
\usepackage{CJK}
\usepackage{fancyhdr}
\usepackage{graphicx}
\usepackage{indentfirst}
\usepackage{latexsym,bm}
\usepackage{mathrsfs}
\usepackage[square, comma, sort&compress, numbers]{natbib}
\allowdisplaybreaks[4]
\newtheorem{theorem}{Theorem}[section]
\newtheorem{lemma}[theorem]{Lemma}
\newtheorem{definition}[theorem]{Definition}
\newtheorem{proposition}[theorem]{Proposition}

\newtheorem{corollary}[theorem]{Corollary}
\newtheorem{remark}[theorem]{Remark}

\usepackage[top=1in,bottom=1in,left=1.25in,right=1.25in]{geometry}
\def\<{\langle}
\def\>{\rangle}
\def\a{\alpha}
\def\b{\beta}

\def\D{\Delta}
\def\di{\diamond}

\def\o{\otimes}

\def\r{\rho}

\date{}
\begin{document}
\renewcommand{\baselinestretch}{1.2}
\renewcommand{\arraystretch}{1.0}
\title{\bf Fundamental theorem of Poisson Hopf module for weak Hopf algebras}
 \date{}
\author {{\bf Daowei Lu$^1$\footnote {Corresponding author:  ludaowei620@126.com}, Dingguo Wang$^2$}\\
{\small $^1$School of Mathematics and Big data, Jining University}\\
{\small Qufu, Shandong 273155, P. R. China}\\
{\small $^2$School of Mathematical Sciences, Qufu Normal University}\\
{\small Qufu, Shandong 273165, P. R. China}
}
 \maketitle
\begin{center}
\begin{minipage}{12.cm}

\noindent{\bf Abstract.} Let $H$ be a weak Hopf algebra with a bijective antipode and $A$ an $H$-comodule Poisson algebra. In this paper, we mainly generalize the fundamental theorem of Poisson Hopf modules to the case of weak Hopf algebras. Besides we will deduce the relative projectivity in the category of Poisson Hopf module.
\\

\\

 \end{minipage}
 \end{center}
 \normalsize\vskip1cm

\section*{Introduction}

Poisson algebras originated from the research of Poisson geometry in the 1970s~\cite{VK,Li} and has
been widely applied in mathematics and physics, such as Poisson
manifolds~\cite{BV}, algebraic geometry~\cite{BLLM}, operads~\cite{GR}, quantization
theory~\cite{Kon}, quantum groups~\cite{CP}, and classical and quantum mechanics~\cite{Od}.

Let $H$ be a Hopf algebra, and $M$ a right $H$-Hopf module. Then there exists an
isomorphism of Hopf modules:
$$M^{coH}\o H\rightarrow M,\ m\o h\mapsto m\cdot h,$$
where $M^{coH}$ denotes the space of coinvariant of $M$. This is the so-called fundamental isomorphism of Hopf modules, which plays an essential role in the theory of Hopf algebras, such as the integral theory, Nichols-Zoller theorem, Galois theory and so on~\cite{Mon}.

Doi~\cite{D} introduced the concept of relative Hopf modules, which is a generalization
of Hopf modules. Moreover, he gave the fundamental theorem of relative Hopf modules, which is a weaker version of the above isomorphism. Explicitly let $H$ be a Hopf algebra, and $A$ a right $H$-comodule algebra. If there exists a right $H$-comodule map $\phi:H\rightarrow A$ which is an algebra map, then for any relative right $(A,H)$-Hopf module $M$, the following isomorphism of relative right $(A,H)$-Hopf modules is obtained
$$M\o_{A^{coH}}A\rightarrow M,\ m\o a\mapsto m\cdot a.$$

In~\cite{Zhang04} Zhang and Zhu generalized this fundamental isomorphism of relative Hopf modules to weak Hopf algebra, and gave the fundamental isomorphism of weak Doi-Hopf modules.
Let $H$ be a Hopf algebra and $A$ a Poisson algebra such that $A$ is a $H$-comodule Poisson algebra. In~\cite{Gu} Gu$\acute{e}$d$\acute{e}$non established the fundamental theorem of Poisson $(A,H)$-Hopf modules, which generalized the fundamental isomorphism to the case of Poisson algebra. Motivated by these results, in this paper, we mainly generalize the fundamental theorem of Poisson Hopf modules to weak Hopf algebras. The paper is organized as follows. In section 1, we will recall basic results on weak Hopf algebras and Poisson algebras. Let $H$ be a weak Hopf algebra and $A$ a Poisson algebra. In section 2, we will introduce the notion of Poisson $(A,H)$-Hopf modules, and prove that for any Poisson $(A,H)$-Hopf modules $M$, $M^{AcoH}$ is a Poisson $A^{AcoH}$-submodule of $M$. In section 3, we will give the first main result, see Theorem 3.4. In section 4, we will establish the fundamental isomorphism of Poisson $(A,H)$-Hopf modules, see Theorem 4.6, and show that $(-\o_BA, (-)^{AcoH})$ is a pair of adjoint functors between the categories $\mathcal{M}_{B}$ and  $\mathcal{M}^H_{\mathcal{P}A}$, where $B=A^{AcoH}$.

\section{Preliminaries}
\def\theequation{1.\arabic{equation}}
\setcounter{equation} {0}

Throughout this paper, let $k$ be a fixed field, and all vector spaces and tensor product are over $k$. For a coalgebra $C$, we will use the Heyneman-Sweedler's notation $\Delta(c)=  c_{1}\otimes c_{2},$
for any $c\in C$ (summation omitted).

Let $H$ be an algebra and a coalgebra. Recall from \cite{Bohm99}, $H$ is a weak bialgebra if it satisfies the following conditions
\begin{align}
&\Delta(xy)=\Delta(x)\Delta(y),\\
&(\Delta\o id)\Delta(1)=(\Delta(1)\o1)(1\o\Delta(1)),\\
&(\Delta\o id)\Delta(1)=(1\o\Delta(1))(\Delta(1)\o1),\\
&\varepsilon(xyz)=\varepsilon(xy_1)\varepsilon(y_2z),\ \varepsilon(xyz)=\varepsilon(xy_2)\varepsilon(y_1z),
\end{align}
for all $x,y,z\in H$. Moreover a weak bialgebra $H$ is called a weak Hopf algebra if there exists a linear map $S:H\rightarrow H$ satisfying
\begin{align}
& x_1S(x_2)=\varepsilon(1_1x)1_2,\\
&S(x_1)x_2=1_1\varepsilon(x 1_2),\\
& S(x_1)x_2S(x_3)=S(x).
\end{align}
The antipode $S$ satisfies the following conditions
\begin{align}
& S(xy)=S(y)S(x),\\
& S(x)_1\o S(x)_2=S(x_2)\o S(x_1),\\
&S(1)=1,\ \varepsilon\circ S=\varepsilon.
\end{align}

Denote $\varepsilon_t(x)=\varepsilon(1_1x)1_2, \varepsilon_s(x)=1_1\varepsilon(x 1_2),$ and the images of $\varepsilon_t$ and $\varepsilon_s$ by $H_t$ and $H_s$ respectively. Now we will list some identities used in this paper as follows:
\begin{align}
&\Delta(z)=1_1x\o1_2,\ \Delta(y)=1_1\o y1_2,\ \hbox{for all\ }y,z\in H_t,\label{1a}\\
&\varepsilon_t(x)=x_1S(x_2),\ \varepsilon_s(x)=S(x_1)x_2,\\
&x_1\o\varepsilon_t(x_2)=1_1 x\o 1_2,\ \varepsilon_t(x_1)\o x_2=S(1_1)\o 1_2 x,\label{1b}\\
&\varepsilon_t(\varepsilon_t(x)y)=\varepsilon_t(xy).\label{1d}
\end{align}

Let $H$ be a weak Hopf algebra. The algebra $A$ is called a right $H$-comodule algebra if $A$ is a right $H$-comodule via $\r:A\rightarrow A\o H,\ a\mapsto a_{(0)}\o a_{(1)}$, and for all $a,b\in A$,
\begin{align}
&(\r\o id)\circ\r=(id\o\Delta)\circ\r,\label{1.2a}\\
&1_{(0)}a \o 1_{(1)}=a_{(0)}\o\varepsilon_t(a_{(1)}),\label{1.2b}\\
&\r(ab)=\r(a)\r(b).\label{1.2c}
\end{align}
It is noticed that (\ref{1.2b}) has an equivalent form
\begin{equation}
1_{(0)}\o1_{(1)1}\o1_{(1)2}=1_{(0)}\o1_11_{(1)}\o1_{2}.\label{1.2b'}
\end{equation}

If $A$ is a right $H$-comodule  algebra, the vector subspace
$$A^{coH}=\{a\in A|a_{(0)}\o a_{(1)}=a_{(0)}\o\varepsilon_t(a_{(1)})\}$$
of $A$ is called the $H$-coinvariants of $A$.

Let $A$ be an $H$-comodule algebra. A vector space $M$ is an $(A, H)$-Hopf module if $M$ is a left $A$ -module and a right $H$-comodule such that
$$(am)_{(0)}\o (am)_{(1)}=a_{(0)}m_{(0)}\o a_{(1)}m_{(1)},$$
for all $a\in A,m\in M$.

\begin{lemma}\cite{Zhang04}
Let $A$ be a right $H$-comodule algebra, and $M$ an $(A,H)$-Hopf module. Then we have
\begin{itemize}
  \item [(1)] for all $m\in M,x\in H$,
  \begin{equation}
  m_{(0)}\varepsilon(m_{(1)}x)=m 1_{(0)}\varepsilon(1_{(1)}x).\label{1.2e}
  \end{equation}
  \item [(2)] $M^{coH}=\{m\in M|\rho(m)=m_{(0)}\o \varepsilon_t(m_{(1)})\}=\{m\in M|\rho(m)=m\cdot1_{(0)}\o 1_{(1)}\}$.
\end{itemize}
\end{lemma}

A Poisson algebra is a commutative associative unitary $k$-algebra $A$ endowed with a bilinear map $\{\cdot,\cdot\}:A\times A\rightarrow A$, called Poisson bracket, providing $A$ with a Lie algebra structure and satisfying the relation
\begin{equation}
\{a,a'a''\}=a'\{a,a''\}+\{a,a'\}a'',\label{1.2d}
\end{equation}
for all $a,a',a''\in A$. A Poisson subalgebra of $A$ is a subalgebra of $A$ which is also a Lie subalgebra of $A$. The Poisson center $A^A$ of $A$ is the Lie-A-invariant elements of $A$, namely
$$A^A=\{a\in A|\{a,A\}=0\}.$$
Obviously $A^A$ is a Poisson subalgebra of $A$.

Let $A$ and $B$ be two Poisson algebras. A homomorphism $f:A\rightarrow B$ of Poisson algebras is a homomorphism of algebras which preserves the brackets, that is,
$$\{f(a),f(a')\}=f(\{a,a'\})$$
for all $a,a'\in A$.

Let $A$ be a Poisson algebra. A vector space $M$ is called a Lie $A$-module with the action
$(m,a)\mapsto m\di a$ if
\begin{equation}
m\di\{a,a'\}=(m\di a)\di a'-(m\di a')\di a,
\end{equation}
for all $a,a'\in A,m\in M$.

A vector space $M$ is a Poisson $A$-module if $M$ is an $A$-module: $(m,a)\mapsto ma$, and a Lie $A$-module: $(m,a)\mapsto m\di a$ satisfying the following compatibility conditions
\begin{align}
&(ma)\di a'=(m\di a')a+m\{a,a'\},\\
&m\di (aa')=(m\di a)a'+(m\di a')a,
\end{align}
for all $a,a'\in A,m\in M$.

We have $1_A\di m=0$ for all $m\in M$. The Poisson algebra $A$ itself is a Poisson $A$-module with $a\di a'=\{a,a'\}$.

Given two Poisson $A$-modules $M$ and $N$, a homomorphism $f:M\rightarrow N$ of Poisson $A$-modules is an $A$-linear map which is also a Lie $A$-linear map from $M$ to $N$. We denote $\mathcal{M}_{\mathcal{P}A}$ the category of Poisson $A$-modules with the morphisms being $A$-linear and Lie $A$-linear. Besides the set of morphisms in $\mathcal{M}_{\mathcal{P}A}$ is denoted by Hom$_{\mathcal{P}A}(M,N)$.

A Poisson submodule of a Poisson $A$-module $M$ is an $A$-submodule of $M$ which is also a Lie $A$-submodule of $M$.

If $M$ is a Poisson $A$-module, the Lie $A$-invariant elements of $M$ is
$$M^A=\{m\in M|m\di a=0, \forall a\in A\},$$
which is a vector subspace of $M$.

\section{The category $\mathcal{M}^H_{\mathcal{P}A}$}
\def\theequation{2.\arabic{equation}}
\setcounter{equation} {0}

%
%
%

\begin{definition}
Let $H$ be a weak Hopf algebra and $A$ a Poisson algebra. We call $A$ a right $H$-comodule Poisson algebra if $A$ is a right $H$-comodule algebra and satisfies the following relation
\begin{equation}
\{a,a'\}_{(0)}\o\{a,a'\}_{(1)}=\{a_{(0)},a'_{(0)}\}\o a_{(1)}a'_{(1)},\label{2a}
\end{equation}
for all $a,a'\in A$.
\end{definition}

If $A$ is a right $H$-comodule Poisson algebra, set
$$A^{AcoH}=\{a\in A|a\in A^A\ \hbox{and}\ a\in A^{coH}\}.$$

\begin{lemma}
Let $A$ be a right $H$-comodule Poisson algebra. Then $A^{coH}$ is an $H$-subcomodule Poisson algebra of $A$.
\end{lemma}

\begin{proof}
First of all , by \cite{Zhang04}, $A^{coH}$ is an $H$-subcomodule algebra of $A$. For all $a,a'\in A^{coH}$,
\begin{align*}
\{a,a'\}_{(0)}\o\varepsilon_t(\{a,a'\}_{(1)})&=\{a_{(0)},a'_{(0)}\}\o\varepsilon_t(a_{(1)}a'_{(1)})\\
&=\{a_{(0)},a'_{(0)}\}\o\varepsilon_t(\varepsilon_t(a_{(1)})a'_{(1)})\\
&\stackrel{(\ref{1d})}=\{a_{(0)},a'_{(0)}\}\o\varepsilon_t(a_{(1)})\varepsilon_t(a'_{(1)})\\
&=\{a_{(0)},a'_{(0)}\}\o a_{(1)}a'_{(1)}\\
&=\{a,a'\}_{(0)}\o \{a,a'\}_{(1)}.
\end{align*}
Hence $\{a,a'\}\in A^{coH}$. That is, $A^{coH}$ is a Lie subalgebra of $A$.
\end{proof}

\begin{definition}
Let $A$ be a right $H$-comodule Poisson algebra. A vector space $M$ is called a Poisson $(A,H)$-Hopf module if $M$ is a Poisson $A$-module and an $(A,H)$-Hopf module such that for all $a\in A,m\in M$,
\begin{equation}
(m\di a)_{(0)}\o (m\di a)_{(1)}=m_{(0)}\di a_{(0)}\o m_{(1)}a_{(1)}.\label{2b}
\end{equation}
\end{definition}

\begin{remark}
The right $H$-comodule Poisson algebra $A$ itself is a Poisson $(A,H)$-Hopf module.
\end{remark}

A Poisson $(A, H)$-Hopf module homomorphism between two Poisson $(A, H)$-Hopf modules is an $A$-linear map, a Lie $A$-linear map and an $H$-colinear map. We denote by $\mathcal{M}^H_{\mathcal{P}A}$ the category of Poisson $(A, H)$-Hopf modules with Poisson $(A, H)$-Hopf module homomorphisms. For $M,N\in\mathcal{M}^H_{\mathcal{P}A}$, we denote by Hom$^H_{\mathcal{P}A}(M,N)$ the vector space of Poisson $(A, H)$-Hopf module homomorphisms from $M$ to $N$.

For any object $M$ in $\mathcal{M}^H_{\mathcal{P}A}$, set
$$M^{AcoH}=\{m\in M|m_{(0)}\o m_{(1)}=m_{(0)}\o \varepsilon_t(m_{(1)}), m\di a=0\}.$$

\begin{lemma}\label{lem1}
Let $H$ be a weak Hopf algebra, $A$ a right $H$-comodule Poisson algebra, and $M$ a Poisson $(A,H)$-Hopf module. Then
\begin{itemize}
  \item [(i)] $M^A$ is an $H$-subcomodule of $M$;
  \item [(ii)] $A^A$ is an $H$-subcomodule Poisson algebra of $A$, and $A^{AcoH}$ is a Poisson subalgebra of $A$.
  \item [(iii)] $M^{AcoH}$ is a Poisson $A^{AcoH}$-submodule of $M$.
\end{itemize}
\end{lemma}

\begin{proof}
(i) For all $m\in M^A,a\in A$,
\begin{align*}
&[(m\di a_{(0)})_{(0)}\o(m\di a_{(0)})_{(1)}]\cdot(1\o S(a_{(1)}))\\
&=[m_{(0)}\di a_{(0)(0)}\o m_{(1)}a_{(0)(1)}]\cdot(1\o S(a_{(1)}))\\
&=m_{(0)}\di a_{(0)}\o m_{(1)}a_{(1)1}S(a_{(1)2})\\
&=m_{(0)}\di a_{(0)}\o m_{(1)}\varepsilon_t(a_{(1)})\\
&=m_{(0)}\di (1_{(0)}a)\o m_{(1)}1_{(1)} \\
&=(m_{(0)}\di 1_{(0)})\cdot a\o m_{(1)}1_{(1)}+(m_{(0)}\di a)\cdot 1_{(0)}\o m_{(1)}1_{(1)} \\
&=(m_{(0)}\di a)\cdot 1_{(0)}\o m_{(1)}1_{(1)}.
\end{align*}
Since $m\in M^A$, we have
$$(m_{(0)}\di a)\cdot 1_{(0)}\o m_{(1)}1_{(1)}=0.$$
Now taking the summands $\{m_{(1)}1_{(1)}\}$ to be linearly independent, we have $(m_{(0)}\di a)\cdot 1_{(0)}=0$ for each summand. Hence
$$(m_{(0)}\di a)\cdot 1_{(0)}\o 1_{(1)}=0.$$
Applying $id\o\varepsilon$ to the above identity, we obtain $m_{(0)}\di a=0$ for each summand $m_{(0)}$. That is, each summand $m_{(0)}$ belongs to $M^A$. So $M^A$ is an $H$-submodule of $M$.

(ii) It is straightforward by (i).

(iii) Clearly  $M^{AcoH}$ is a Lie $A^{AcoH}$-submodule of $M$.
For all $m\in M^{AcoH},a\in A$,
\begin{align*}
\r(m\cdot a)&=(m\cdot a)_{(0)}\o(m\cdot a)_{(1)}\\
&=m_{(0)}\cdot a_{(0)}\o m_{(1)} a_{(1)}\\
&=m\cdot1_{(0)} 1'_{(0)}a_{(0)}\o 1_{(1)} 1'_{(1)}\\
&=(m\cdot a)\cdot 1_{(0)}\o 1_{(1)}.
\end{align*}
Obviously $m\cdot a\in M^A$. Thus $m\cdot a\in M^{AcoH}$. Therefore $M^{AcoH}$ is an $A^{AcoH}$-submodule of $M$.
\end{proof}

\section{The first main result}
\def\theequation{3.\arabic{equation}}
\setcounter{equation} {0}

In what follows, we will say that a vector space $M$ is an $(\underline{A},H)$-comodule if $M$ is a Lie $A$-module, a right $H$-comodule and the relation (\ref{2b}) is satisfied. We denote by $_{\underline{A}}\mathcal{M}$ the category of Lie $A$-modules with Lie $A$-linear maps, and by $_{\underline{A}}\mathcal{M}^H$ the category of $(\underline{A},H)$-comodule with Lie $\underline{A}$-linear and $H$-colinear maps.

Throughout this section we always assume that the $H$-comodule Poisson algebra $A$ satisfies the following identity:
\begin{equation}\label{Pro}
\{a,1_{(0)}\}\o 1_{(1)}=0,\ \hbox{for all}\ a\in A.
\end{equation}

For a Poisson $A$-module $N$, define the subspace $N\hat{\o} H$ of $N\o H$ as follows:
\begin{align*}
N\hat{\o} H=\Big\{&\sum_i n_i\o h_i\in N\o H|\sum_i n_i\cdot 1_{(0)}\o h_i1_{(1)}=\sum_i n_i\o h_i,\hbox{and}\\
&\sum_i n_i\di1_{(0)}\o h_i1_{(1)}=0\Big\}.
\end{align*}
Actually $N\hat{\o} H$ could also be seen as a subspace of $(N\o H)\cdot(1_{(0)}\o1_{(1)})$ whose elements satisfy $\sum_i n_i\di1_{(0)}\o h_i1_{(1)}=0$.

\begin{lemma}\label{3a}
(i) Let $N$ be a Poisson $A$-module. Then $N\hat{\o} H$ is an $(\underline{A},H)$-comodule with the Lie $A$-action and $H$-coaction given by
\begin{align*}
&(n\o h)\di a=n\di a_{(0)}\o ha_{(1)},\\
&(n\o h)_{(0)}\o(n\o h)_{(1)}=n\o h_1\o h_2,
\end{align*}
for all $a\in A, n\o h\in N\hat{\o} H$.

(ii) Furthermore, if $H$ is commutative, then $N\hat{\o} H$ is a Poisson $(A, H)$-Hopf module with the $A$-action given by
$$(n\o h)\cdot a=n\cdot a_{(0)}\o ha_{(1)},$$
for all $a\in A, n\o h\in N\hat{\o} H$.
\end{lemma}

\begin{proof}
For all $a\in A, n\o h\in N\hat{\o} H$, first of all, it is easy to see that $(n\o h)\di a\in N\hat{\o} H$, and
\begin{align*}
&n\o h_1\o h_2\\
&=n\cdot1_{(0)}\o h_11_{(1)1}\o h_21_{(1)2}\\
&=n\cdot1_{(0)}\o h_11_11_{(1)}\o h_21_{2}\\
&=n\cdot1_{(0)}\o h_11_1\o h_2,
\end{align*}
that is, $(n\o h)_{(0)}\o(n\o h)_{(1)}\in N\hat{\o} H\o H$.

The rest of the proof is similar to that given in \cite[Lemma 2.1]{Gu}.
\end{proof}

From Lemma \ref{3a}, if $H$ is commutative, then $A\o H$ is a Poisson $(A, H)$-Hopf module: the $H$-coaction is $id_A\o\D_H$, the $A$-action and Lie $A$-action are given by
$$(a\o h)\cdot a'=a\cdot a'_{(0)}\o ha'_{(1)},\ (a\o h)\di a'=\{a,a'_{(0)}\}\o ha'_{(1)},$$
for all $a,a'\in A, h\in H$. Note that the identity (\ref{Pro}) implies $(a\o h)\di 1=\{a,1_{(0)}\}\o h1_{(1)}=0$.

\begin{lemma}\label{3b}
Let $H$ be a weak Hopf algebra with $H_s$ being commutative.

(i) Let $M$ be an $(\underline{A},H)$-comodule and $N$ a Poisson $A$-module. There exists a linear isomorphism
$$\gamma:Hom^H_{\underline{A}}(M,N\hat{\o} H)\rightarrow Hom_{\underline{A}}(M,N),\ f\mapsto(id\o\varepsilon)\circ f.$$

(ii) Let $H$ be commutative, $M$ a Poisson $(A, H)$-Hopf module and $N$ a Poisson $A$-module. There exists a linear isomorphism
$$\gamma:Hom^H_{\mathcal{P}A}(M,N\hat{\o} H)\rightarrow Hom_{\mathcal{P}A}(M,N),\ f\mapsto(id\o\varepsilon)\circ f.$$
\end{lemma}

\begin{proof}
(i) By Lemma \ref{3a}, $N\hat{\o} H$ is an $(\underline{A},H)$-comodule. For $f\in Hom^H_{\underline{A}}(M,N\hat{\o} H)$, we write $f(m)=\sum_i(n_i\o h_i)$. Then $\gamma(f)(m)=\sum_in_i\varepsilon(h_i)$. For $a\in A$, we have
$$f(m\di a)=f(m)\di a=\sum_in_i\di a_{(0)}\o h_ia_{(1)}.$$
Since
\begin{align*}
\gamma(f)(m\di a)&=\sum_in_i\di a_{(0)}\varepsilon(h_ia_{(1)})\\
&=\sum_in_i\di 1_{(0)}a_{(0)}\varepsilon(h_i1_2)\varepsilon(1_11_{(1)}a_{(1)})\\
&\stackrel{(\ref{1.2b'})}=\sum_in_i\di 1_{(0)}a_{(0)}\varepsilon(h_i1_{(1)2})\varepsilon(1_{(1)1}a_{(1)})\\
&=\sum_in_i\di 1_{(0)(0)}a_{(0)}\varepsilon(h_i1_{(1)})\varepsilon(1_{(0)(1)}a_{(1)})\\
&=\sum_in_i\di (1_{(0)}a)_{(0)}\varepsilon(h_i1_{(1)})\varepsilon((1_{(0)}a)_{(1)})\\
&=\sum_in_i\di (1_{(0)}a)\varepsilon(h_i1_{(1)})\\
&=\sum_i(n_i\di 1_{(0)})\cdot a\varepsilon(h_i1_{(1)})+\sum_i(n_i\di a)\cdot 1_{(0)}\varepsilon(h_i1_{(1)})\\
&=\sum_i(n_i\cdot 1_{(0)})\di a\varepsilon(h_i1_{(1)})+n_i\{a,1_{(0)}\}\varepsilon(h_i1_{(1)})\\
&=(\sum_in_i\varepsilon(h_i))\di a,
\end{align*}
where the last two identities hold due to $\sum_in_i\di 1_{(0)}\o h_i1_{(1)}=0$ and the identity (\ref{Pro}), we obtain that $\gamma(f)$ is Lie $A$-linear.

For $g\in Hom_{\underline{A}}(M,N)$, define $\gamma'(g):M\rightarrow N\o H$ by
$$\gamma'(g)(m)=g(m_{(0)})\cdot 1_{(0)}\o m_{(1)}1_{(1)},$$
for all $m\in M$. Moreover
\begin{align*}
&(g(m_{(0)})\cdot 1_{(0)})\di 1'_{(0)}\o m_{(1)}1_{(1)}1'_{(1)}\\
&=(g(m_{(0)})\di 1'_{(0)})\cdot 1_{(0)}\o m_{(1)}1_{(1)}1'_{(1)}+g(m_{(0)})\cdot \{1_{(0)}, 1'_{(0)}\}\o m_{(1)}1_{(1)}1'_{(1)}\\
&=g(m_{(0)}\di 1'_{(0)})\cdot 1_{(0)}\o m_{(1)}1'_{(1)}1_{(1)}+g(m_{(0)})\cdot \{1, 1'\}_{(0)}\o m_{(1)}\{1, 1'\}_{(1)}\\
&=g((m\di 1)_{(0)})\cdot 1_{(0)}\o (m\di 1)_{(1)}1_{(1)}\\
&=0,
\end{align*}
thus $Im\gamma'(g)\subseteq N\hat{\o} H$.

By a similar argument as in \cite{Gu}, we get that $\gamma'(g)$ is a morphism in $Hom^H_{\underline{A}}(M,N\hat{\o} H)$, and $\gamma$ is a bijective.

(ii) By Lemma \ref{3a}, $N\hat{\o} H$ is a Poisson $(A,H)$-module. One can easily check that $\gamma(f)$ and $\gamma'(g)$ are $A$-linear.
\end{proof}

\begin{corollary}\label{3c}
(i) Let $H$ be a weak Hopf algebra and $H_s\subseteq Z(H)$. If $N$ is an injective Poisson $A$-module, then $N\hat{\o} H$ is an injective $(\underline{A}, H)$-comodule.

(ii) Let $H$ commutative. If $N$ is an injective Poisson $A$-module, then $N\hat{\o} H$ is an injective Poisson $(A, H)$-Hopf module.
\end{corollary}

\begin{proof}
The result is an immediate consequence of Lemma \ref{3a}.
\end{proof}

\begin{theorem}\label{3d}
Let $H$ be a weak Hopf algebra satisfying $\Delta(1)=\Delta^{op}(1)$ and $H_s\subseteq Z(H)$. Assume that there exists an $H$-colinear map $\phi:H\rightarrow A^A$ such that $\phi(1)=1$ and $1_{(0)}\phi(1_{(1)})=1$.

(i) If $\phi$ is an algebra map, then every Poisson $(A, H)$-Hopf module which is injective as a Lie $A$-module is an injective $(\underline{A}, H)$-comodule.

(ii) If $H$ is commutative and $\phi$ is an algebra map, then every Poisson $(A, H)$-Hopf module which is injective as a Poisson $A$-module is an injective Poisson $(A, H)$-Hopf module.
\end{theorem}

\begin{proof}
(i) Let $M$ be a Poisson $(A, H)$-Hopf module. Define $\lambda:M\hat{\o}H\rightarrow M$ by
$$\lambda(m\o h)=m_{(0)}\cdot\phi(S(m_{(1)})h),$$
for all $m\in M,h\in H$. Since
\begin{align*}
(\lambda\circ\r)(m)&=\lambda(m_{(0)}\o m_{(1)})\\
&=m_{(0)}\cdot\phi(S(m_{(1)1})m_{(1)2})\\
&=m_{(0)}\cdot\phi(\varepsilon_s(m_{(1)}))\\
&=m_{(0)}\cdot\varepsilon(m_{(1)}1_2)\phi(1_1)\\
&=m\cdot1_{(0)}\varepsilon(1_{(1)}1_2)\phi(1_1)\\
&=m\cdot1_{(0)}\varepsilon(1_{(1)}1_1)\phi(1_2)\\
&=m\cdot1_{(0)}\varepsilon(1_{(1)1})\phi(1_{(1)2})\\
&=m\cdot1_{(0)}\phi(1_{(1)})\\
&=m,
\end{align*}
thus $\lambda\circ\r=id_M$. And
\begin{align*}
&\lambda(m\o h)_{(0)}\o\lambda(m\o h)_{(1)}\\
&=[m_{(0)}\cdot\phi(S(m_{(1)})h)]_{(0)}\o[m_{(0)}\cdot\phi(S(m_{(1)})h)]_{(1)}\\
&=m_{(0)}\cdot\phi(S(m_{(1)2})h)_{(0)}\o m_{(1)1}\phi(S(m_{(1)2})h)_{(1)}\\
&=m_{(0)}\cdot\phi(S(m_{(1)2})_1h_1)\o m_{(1)1}S(m_{(1)2})_2h_{2}\\
&=m_{(0)}\cdot\phi(S(m_{(1)3})h_1)\o m_{(1)1}S(m_{(1)2})h_{2}\\
&=m_{(0)}\cdot\phi(S(m_{(1)2})h_1)\o \varepsilon_t(m_{(1)1})h_2\\
&\stackrel{(\ref{1.2b})}=m_{(0)}\cdot\phi(S(1_2m_{(1)})h_1)\o S(1_1)h_2\\
&=m_{(0)}\cdot\phi(S(m_{(1)})1_1h_1)\o 1_2h_2\\
&=m_{(0)}\cdot\phi(S(m_{(1)})h_1)\o h_2\\
&=\lambda((m\o h)_{(0)})\o(m\o h)_{(1)},
\end{align*}
hence $\lambda$ is $H$-colinear. Finally on one hand,
\begin{align*}
\lambda((m\o h)\di a)&=\lambda(m\di a_{(0)}\o ha_{(1)})\\
&=(m\di a_{(0)})_{(0)}\cdot\phi(S((m\di a_{(0)})_{(1)})ha_{(1)})\\
&=(m_{(0)}\di a_{(0)})\cdot\phi(S(m_{(1)} a_{(1)1})ha_{(1)2})\\
&=(m_{(0)}\di a_{(0)})\cdot\phi(S(a_{(1)1})S(m_{(1)})ha_{(1)2})\\
&=(m_{(0)}\di a_{(0)})\cdot\phi(S(m_{(1)})h\varepsilon_s(a_{(1)}))\\
&=(m_{(0)}\di a_{(0)})\cdot\phi(S(m_{(1)})h1_1)\varepsilon(a_{(1)}1_2)\\
&=(m_{(0)}\di (a_{(0)}1_{(0)}))\cdot\phi(S(m_{(1)})h1_2)\varepsilon(a_{(1)}1_{(1)}1_1)\\
&=(m_{(0)}\di (a_{(0)}1_{(0)}))\cdot\phi(S(m_{(1)})h1_{(1)2})\varepsilon(a_{(1)}1_{(1)1})\\
&=(m_{(0)}\di (a1_{(0)}))\cdot\phi(S(m_{(1)})h1_{(1)})\\
&=(m_{(0)}\di (1_{(0)}a))\cdot\phi(S(m_{(1)})S(1_{(1)})h)\\
&=(m_{(0)}\di 1_{(0)})\cdot a\phi(S(m_{(1)}1_{(1)})h)+(m_{(0)}\di a)\cdot 1_{(0)}\phi(S(m_{(1)}1_{(1)})h)\\
&=(m_{(0)}\di a)\cdot 1_{(0)}\phi(1_{(1)})\phi(S(m_{(1)})h)\\
&=(m_{(0)}\di a)\cdot \phi(S(m_{(1)})h),
\end{align*}
where the tenth identity is true by \cite[Lemma 2.5]{Zhao}.
On the other hand,
\begin{align*}
\lambda(m\o h)\di a&=(m_{(0)}\cdot\phi(S(m_{(1)}h)))\di a\\
&=(m_{(0)}\di a)\cdot\phi(S(m_{(1)}h))+m_{(0)}\cdot\{\phi(S(m_{(1)}h)), a\}\\
&=(m_{(0)}\di a)\cdot\phi(S(m_{(1)}h)).
\end{align*}
It implies that $\lambda$ is a Lie $A$-linear map. Therefore $\lambda$ is a morphism in $_{\underline{A}}\mathcal{M}^H$. By Corollary \ref{3c}, $M\hat{\o }H$ is an injective $(\underline{A}, H)$-comodule. It is obvious that $\r_M:M\rightarrow M\hat{\o} H$ is $H$-colinear and Lie $A$-linear, and $\r_M$
is an injective map. So $M$ is a direct summand of $M\hat{\o }H$ as an $(\underline{A}, H)$-comodule. This implies that $M$  is an injective $(\underline{A}, H)$-comodule, being a direct summand of the injective $(\underline{A}, H)$-comodule $M\hat{\o }H$.


(ii) By a similar computation as above, we obtain that $\lambda$ given in (1) is $A$-linear. Thus $\lambda$ is a homomorphism of Poisson $(A, H)$-Hopf modules. By the same arguments as above, we conclude that an injective Lie $A$-module is also an injective Poisson $(A,H)$-Hopf module.
\end{proof}

\begin{remark}
(i) Compared with \cite[Theorem 2.4]{Gu}, the condition that $\phi$ is an algebra map is necessary in Theorem \ref{3d}.

(ii) In \cite{Jia} the authors introduced the definition of weak total integral of weak Hopf group coalgebras. When reduced to weak Hopf algebras, the total integral $\phi:H\rightarrow A$ is an $H$-comodule map satisfying $1_{(0)}\phi(S(1_{(1)}))=1$, where $A$ is a right $H$-comodule algebra. Now when $\Delta(1)=\Delta^{op}(1)$, this identity could be rewritten as $1_{(0)}\phi(1_{(1)})=1$, which makes sense in Theorem \ref{3d}.
\end{remark}

\section{The second main result}
\def\theequation{4.\arabic{equation}}
\setcounter{equation} {0}

\begin{lemma}\label{lem2}
Set $B=A^{AcoH}$. Let $M$ be a Poisson $B$-module. Then $M\o_BA$ is a Poisson $(A,H)$-Hopf module, where the $A$-action is trivial, the Lie $A$-action is given by
$$(m\o_B a)\di a'=m\o_B a\di a',$$
for all $a,a'\in A,m\in M$, and the $H$-coaction is given by
$$\r_{M\o_BA}(m\o a)=m\o_B a_{(0)}\o a_{(1)}.$$
\end{lemma}

\begin{lemma}\label{lem3}
Set $B=A^{AcoH}$. Let $M$ be a Poisson $(A,H)$-Hopf module. Then the linear map $\a:M^{AcoH}\o_BA\rightarrow M,\ m\o_Ma\mapsto m\cdot a$ is a homomorphism of Poisson $(A,H)$-Hopf modules.
\end{lemma}

Consider $H$ as a right $H$-comodule via $\Delta$. Let $\phi:H\rightarrow A$ be a morphism of right $H$-comodule and and a morphism of algebra such that $\phi(1)=1$. For any $M\in\mathcal{M}^H_{\mathcal{P}A}$, define the $k$-linear map
$$p_M:M\rightarrow M,\ m\mapsto m_{(0)}\cdot\phi(S(m_{(1)})).$$
Since $\phi$ is a morphism of right $H$-comodule, clearly
\begin{equation}\label{2.1}
1_{(0)}\o1_{(1)}=\phi(1_1)\o1_2.
\end{equation}

\begin{lemma}\label{lem4}
With $M$ and $p_M$ defined as above. We have
\begin{enumerate}
  \item [(i)] $p_M(M)=M^{coH}$,
  \item [(ii)] $p_M\circ p_M=p_M$.
\end{enumerate}
\end{lemma}

\begin{proof}
(i) Following the proof of \cite[Theorem 2.2]{Zhang04}, we have $p_M(M)\subseteq M^{coH}$. Conversely for $m\in M^{coH}$,
\begin{align*}
p_M(m)&=m_{(0)}\cdot\phi(S(m_{(1)}))\\
&=m\cdot1_{(0)}\phi(S(1_{(1)}))\\
&\stackrel{(\ref{2.1})}=m\cdot\phi(1_{1})\phi(S(1_{2}))\\
&=m.
\end{align*}
So $m\in p_M(M)$, that is, $M^{coH}\subseteq p_M(M)$.

(ii) For all $m\in M$,
\begin{align*}
p_M(p_M(m))&=(m_{(0)}\cdot\phi(S(m_{(1)})))_{(0)}\phi(S((m_{(0)}\cdot\phi(S(m_{(1)})))_{(1)}))\\
&=m_{(0)}\cdot\phi(S(m_{(1)3}))\phi(S(m_{(1)1}S(m_{(1)2})))\\
&=m_{(0)}\cdot\phi(S(m_{(1)2}))\phi(S(\varepsilon_t(m_{(1)1})))\\
&=m_{(0)}\cdot\phi(S(\varepsilon_t(m_{(1)1})m_{(1)2}))\\
&=m_{(0)}\cdot\phi(S(m_{(1)}))=p_M(m).
\end{align*}
\end{proof}

Let $M\in\mathcal{M}^H_{\mathcal{P}A}$. For all $a\in A$ and $m\in M$, put
$$m\di' a=p_M(m\di a).$$

\begin{lemma}\label{lem5}
Let $H$ be a weak Hopf algebra. Let $M\in\mathcal{M}^H_{\mathcal{P}A}$. Suppose that there exists a right $H$-colinear algebra map $\phi:H\rightarrow A^A$, then for all $a,c\in A,m\in M$,
\begin{itemize}
  \item [(i)] $p_M(m\cdot a)=p_M(m)\cdot p_A(a),$
  \item [(ii)] $p_M(m)\di'a=p_M(m\di a)=m\di'a,$
  \item [(iii)] $m\di'\{a,c\}=(m\di'a)\di'c-(m\di'c)\di'a,$
  \item [(iv)] $p_M(m)\di a=(p_M(m)\di' a_{(0)})\cdot \phi(S^2(a_{(1)}))$,
  \item [(v)] $p_M(m_{(0)})\cdot\phi(m_{(1)})=m$.
\end{itemize}
\end{lemma}

\begin{proof}
We only prove (ii), (iv) and (v).

(ii) For all $a\in A,m\in M$,
\begin{align*}
p_M(m)\di' a&=p_M(p_M(m)\di a)\\
&=p_M((m_{(0)}\cdot\phi(S(m_{(1)})))\di a)\\
&=p_M(m_{(0)}\cdot\{\phi(S(m_{(1)})), a\})+p_M((m_{(0)}\di a)\cdot\phi(S(m_{(1)})))\\
&=p_M((m_{(0)}\di a)\cdot\phi(S(m_{(1)})))\\
&=p_M(m_{(0)}\di a)\cdot p_A(\phi(S(m_{(1)})))\\
&=(m_{(0)}\di a_{(0)})\cdot\phi(S(m_{(1)1}a_{(1)}))\phi(S(m_{(1)2}))_{(0)}\phi(S(\phi(S(m_{(1)2}))_{(1)}))\\
&=(m_{(0)}\di a_{(0)})\cdot\phi(S(m_{(1)1}a_{(1)}))\phi(S(m_{(1)3}))\phi(S^2(m_{(1)2}))\\
&=(m_{(0)}\di a_{(0)})\cdot\phi(S(m_{(1)1}a_{(1)}))\phi(S(\varepsilon_s(m_{(1)2})))\\
&\stackrel{(\ref{1.2b})}=(m_{(0)}\di a_{(0)})\cdot\phi(S(m_{(1)}1_1a_{(1)}))\phi(S^2(1_2))\\
&=(m_{(0)}\di a_{(0)})\cdot\phi(S(m_{(1)}a_{(1)}))\phi(S(1_1S(1_2)))\\
&=(m\di a)_{(0)}\cdot\phi(S((m\di a)_{(1)}))\\
&=p_M(m\di a)=m\di' a.
\end{align*}

(iv) For all $a\in A,m\in M$,
\begin{align*}
&(p_M(m)\di' a_{(0)})\cdot\phi(S^2(a_{(1)}))\\
&\stackrel{(ii)}=(m\di' a_{(0)})\cdot\phi(S^2(a_{(1)}))\\
&=p_M(m\di a_{(0)})\cdot\phi(S^2(a_{(1)}))\\
&=(m_{(0)}\di a_{(0)})\cdot\phi(S(m_{(1)}a_{(1)1}))\phi(S^2(a_{(1)2}))\\
&=(m_{(0)}\di a_{(0)})\cdot\phi(S(m_{(1)})S(a_{(1)1}S(a_{(1)2})))\\
&=(m_{(0)}\di a_{(0)})\cdot\phi(S(m_{(1)}\varepsilon_t(a_{(1)})))\\
&\stackrel{(\ref{1.2b})}=(m_{(0)}\di 1_{(0)}a)\cdot\phi(S(m_{(1)}1_{(1)}))\\
&=(m_{(0)}\di 1_{(0)})\cdot a\phi(S(m_{(1)}1_{(1)}))+(m_{(0)}\di a)\cdot 1_{(0)}\phi(S(m_{(1)}1_{(1)}))\\
&=(m_{(0)}\di a)\cdot 1_{(0)}\phi(S(m_{(1)}1_{(1)}))\\
&=(m_{(0)}\di a)\cdot\phi(1_{1}S(m_{(1)}1_{2}))\\
&=(m_{(0)}\di a)\cdot\phi(S(m_{(1)}))\\
&=p_M(m)\di a.
\end{align*}

(v) For all $m\in M$,
\begin{align*}
p_M(m_{(0)})\cdot\phi(m_{(1)})&=m_{(0)}\cdot\phi(S(m_{(1)1}))\phi(m_{(1)2})\\
&=m_{(0)}\cdot\phi(S(m_{(1)1})m_{(1)2})\\
&=m_{(0)}\cdot\phi(1_{1})\varepsilon(m_{(1)}1_{2})\\
&=m_{(0)}\cdot 1_{(0)}\varepsilon(m_{(1)}1_{(1)})\\
&=m.
\end{align*}
\end{proof}

It follows from Lemma \ref{lem5} (v) that every Poisson $(A,H)$-Hopf module $M$ is a Lie $A$-module under the new Lie action $\di'$, and $M^{coH}$ is a Lie $A$-submodule of $M$. Indeed for all $m\in M^{coH},a\in A$,
\begin{align*}
&(m\di'a)_{(0)}\o(m\di'a)_{(1)}=p_M(m\di a)_{(0)}\o p_M(m\di a)_{(1)}\\
&=[(m_{(0)}\di a_{(0)})\cdot\phi(S(m_{(1)}a_{(1)}))]_{(0)}\o [(m_{(0)}\di a_{(0)})\cdot\phi(S(m_{(1)}a_{(1)}))]_{(1)}\\
&=(m_{(0)}\di a_{(0)})\cdot\phi(S(m_{(1)3}a_{(1)3}))\o m_{(1)1} a_{(1)1}S(m_{(1)2}a_{(1)2})\\
&=(m_{(0)}\di a_{(0)})\cdot\phi(S(m_{(1)2}a_{(1)2}))\o \varepsilon_t(m_{(1)1} a_{(1)1})\\
&\stackrel{(\ref{1b})}=(m_{(0)}\di a_{(0)})\cdot\phi(S(1_2m_{(1)}a_{(1)}))\o S(1_1)\\
&=((m\cdot1_{(0)})\di a_{(0)})\cdot\phi(S(1_21_{(1)}a_{(1)}))\o S(1_1)\\
&=((m\cdot1_{(0)})\di a_{(0)})\cdot\phi(S(1_{(1)}a_{(1)})1_1)\o 1_2\\
&=((m\cdot1_{(0)})\di a_{(0)})\cdot\phi(S(1_{(1)}a_{(1)}))1'_{(0)}\o 1'_{(1)}\\
&=(m_{(0)}\di a_{(0)})\cdot\phi(S(m_{(1)}a_{(1)}))1'_{(0)}\o 1'_{(1)}\\
&=(m\di' a)\cdot1'_{(0)}\o 1'_{(1)},
\end{align*}
where $1_{(0)}\o 1_{(1)}=1'_{(0)}\o 1'_{(1)}$.

Since the right $H$-comodule Poisson algebra $A$ belongs to the category $\mathcal{M}^H_{\mathcal{P}A}$, $A$ is also a Lie $A$-module under $\di'$, namely, for all $a,b\in A$,
$$a\di'b=p_A(a\di b)=p_A([a,b]),$$
and $A^{coH}$ is a Lie $A$-submodule of $A$ under $\di'$.

\begin{lemma}\label{lem6}
Let $\phi:H\rightarrow A^A$ be a right $H$-colinear algebra map, and $M\in\mathcal{M}^H_{\mathcal{P}A}$.
\begin{itemize}
  \item [(1)] If the Lie action $\di'$ on $M^{coH}$ is trivial, then $M^{AcoH}=M^{coH}$.
  \item [(2)] If the Lie action $\di'$ on $A^{coH}$ is trivial, then $A^{AcoH}=A^{coH}$.
\end{itemize}
\end{lemma}

\begin{proof}
(1) Clearly $M^{AcoH}\subseteq M^{coH}$. For all $a\in A,m\in M^{coH}$, by Lemma \ref{lem5} (4), we have
\begin{align*}
m\di a=p_M(m)\di a=(p_M(m)\di' a_{(0)})\cdot \phi(S^2(a_{(1)}))=0.
\end{align*}

(2) It is a direct consequence of (1).
\end{proof}

\begin{theorem}\label{thm7}
Set $B=A^{AcoH}.$ Let $M$ be a Poisson $(A, H)$-Hopf module, and $\phi:H\rightarrow A^A$ a right $H$-colinear algebra map. Suppose $M^{coH}$ and $A^{coH}$ are trivial Lie $A$-modules under $\di'$. Then the linear map $\a:M^{AcoH}\o_B A \rightarrow M,\ m\o a\mapsto m\cdot a$ is an isomorphism of Poisson $(A, H)$-Hopf modules.
\end{theorem}

\begin{proof}
By Lemma \ref{lem3}, $\a$ is a homomorphism of Poisson $(A,H)$-Hopf modules. Let $a\in A,m\in M$, by Lemma \ref{lem6}, we have $p_M(m)\in M^{AcoH}$. So the following linear map is well defined:
$$\b:M\rightarrow M^{AcoH}\o_BA,\ m\mapsto p_M(m_{(0)})\o_B\phi(m_{(1)}).$$
First of all, by Lemma \ref{lem5} (5), $\a\circ\b=id_M$. Then for all $a\in A, m\in M^{AcoH}$,
\begin{align*}
\b(\a(m\o a))&=m_{(0)}\cdot a_{(0)}\phi(S(m_{(1)1}a_{(1)1}))\o_B\phi(m_{(1)2}a_{(1)2})\\
&=m\cdot 1_{(0)}a_{(0)}\phi(S(1_{(1)1}a_{(1)1}))\o_B\phi(1_{(1)2}a_{(1)2})\\
&=m\cdot a_{(0)}\phi(S(a_{(1)1}))\o_B\phi(a_{(1)2})\\
&=m\o_B a_{(0)}\phi(S(a_{(1)1}))\phi(a_{(1)2})\ \ \hbox{by Lemma \ref{lem6}}\\
&=m\o_B a_{(0)}\phi(S(a_{(1)1}))\phi(a_{(1)2})\\
&=m\o_B a_{(0)}\phi(\varepsilon_s(a_{(1)}))\\
&=m\o_B a_{(0)}\phi(1_1)\varepsilon(a_{(1)}1_2)\\
&=m\o_B a_{(0)}1_{(0)}\varepsilon(a_{(1)}1_{(1)})\\
&=m\o_B a.
\end{align*}
That is, $\b\circ\a=id_{M^{AcoH}\o_BA}$. The proof is completed.
\end{proof}

Let $A$ be a right $H$-comodule Poisson algebra. For any morphism $f:M\rightarrow N$ in $\mathcal{M}^H_{\mathcal{P}A}$ and $m\in M^{AcoH}$,
\begin{align*}
f(m)_{(0)}\o f(m)_{(1)}&=f(m_{(0)})\o m_{(1)}\\
&=f(m_{(0)})\o \varepsilon_t(m_{(1)})\\
&=f(m)_{(0)}\o \varepsilon_t(f(m)_{(1)}),
\end{align*}
and for all $A\in A$,
$$f(m)\di a=f(m\di a)=0,$$
that is, $f(m)\in N^{AcoH}$. This gives rise to a functor
$$G=(-)^{AcoH}:\mathcal{M}^H_{\mathcal{P}A}\rightarrow \mathcal{M}_{B},\ M\mapsto M^{AcoH}.$$
By Lemma \ref{lem2}, we also have a functor
$$F:\mathcal{M}_{B}\rightarrow\mathcal{M}^H_{\mathcal{P}A},\ M\mapsto M\o_BA.$$

\begin{proposition}
Let $M\in\mathcal{M}^H_{\mathcal{P}A}, N\in\mathcal{M}_{B}$. There exists a functorial isomorphism of Poisson $(A,H)$-Hopf modules
\begin{align*}
\psi:Hom^H_{\mathcal{P}A}(N\o_B&A,M)\rightarrow Hom_B(N,M^{AcoH})\\
&f\mapsto\psi(f):N\rightarrow M^{AcoH},\ n\mapsto f(1\o n).
\end{align*}
Thus the functors $F$ and $G$ form an adjoint pair with unit and counit
\begin{align*}
&\eta_N:N\rightarrow (N\o_B A)^{AcoH},\ n\mapsto 1\o n,\\
&\epsilon_M:M^{AcoH}\o_B A\rightarrow M,\ m\o a\mapsto m\cdot a.
\end{align*}
\end{proposition}

\begin{proof}
For all $N\in\mathcal{M}_{B}$ and $n\in N$, since $f$ is a morphism of $H$-comodule,
\begin{align*}
&f(n\o1)_{(0)}\o\varepsilon_t(f(n\o1)_{(1)})=f((n\o1)_{(0)})\o\varepsilon_t((n\o1)_{(1)})\\
&=f(n\o1_{(0)})\o\varepsilon_t(1_{(1)})=f(n\o1_{(0)})\o1_{(1)}\\
&=f(n\o1)_{(0)}\o f(n\o1)_{(1)}.
\end{align*}
And obviously for all $a\in A$, $f(n\o 1)\di a=f(n\o 1\di a)=0$, we have $f(n\o 1)\in M^{AcoH}.$
It is straightforward to verify that $\psi(f)$ is a morphism of $B$-module. Hence $\psi$ is well defined. Now define
\begin{align*}
\psi':Hom_B(N,&M^{AcoH})\rightarrow Hom^H_{\mathcal{P}A}(N\o_BA,M)\\
&g\mapsto\psi'(g):N\rightarrow M^{AcoH},\ n\o a\mapsto g(n)\cdot a.
\end{align*}
It is obvious that $\psi'(g)$ is a morphism of both $A$-module and Lie $A$-module. And
\begin{align*}
&\psi'(g)(n\o a)_{(0)}\o\psi'(g)(n\o a)_{(1)}\\
&=(g(n)\cdot a)_{(0)}\o(g(n)\cdot a)_{(1)}\\
&=g(n)_{(0)}\cdot a_{(0)}\o g(n)_{(1)}\cdot a_{(1)}\\
&=g(n)\cdot 1_{(0)}a_{(0)}\o 1_{(1)}a_{(1)}\\
&=g(n)\cdot a_{(0)}\o a_{(1)},
\end{align*}
hence $\psi'(g)$ is a morphism of $H$-comodule, and therefore $\psi'$ is well defined. More $\psi$ and $\psi'$ are mutual inverse, and the verification is trivial. The proof is completed.
\end{proof}


\section*{Acknowledgement}

The authors are very grateful to the referee for his/her valuable comments. This work was supported by the NNSF of China (Nos. 12271292, 11901240).

%

\end{document}